\documentclass[a4paper,11pt,twoside,reqno]{amsart}

\usepackage[utf8]{inputenc}
\usepackage{amsmath, amsfonts, amssymb,amsthm}
\usepackage[mathscr]{eucal} 
\usepackage[pdftex]{graphicx}
\usepackage{color}

\usepackage[T1]{fontenc}
\usepackage[sc]{mathpazo}
\usepackage[all]{xy}

\usepackage{enumerate}

\usepackage{tikz-cd}

\newtheorem{theorem}{Theorem}[section]
\newtheorem*{theorem*}{Theorem}
\newtheorem{proposition}[theorem]{Proposition}
\newtheorem{corollary}[theorem]{Corollary}
\newtheorem{lemma}[theorem]{Lemma}

\newtheorem{question}[theorem]{Question}

\theoremstyle{definition}
  \newtheorem{definition}[theorem]{Definition}

\theoremstyle{remark}
  \newtheorem{remark}[theorem]{Remark}
  \newtheorem{example}[theorem]{Example}

\numberwithin{equation}{section}

\title[On the duality between \textsc{cmc} surfaces in $\mathbb{E}(\kappa,\tau)$ and $\mathbb{L}(\kappa,\tau)$]{On the conformal duality between surfaces with constant mean curvature in $\mathbb{E}(\kappa,\tau)$ and $\mathbb{L}(\kappa,\tau)$}

\author{José M. Manzano}
\address{Department of Mathematics \\
King's College London \\
Strand WC2R 2LS London} 
\email{manzanoprego@gmail.com}

\thanks{This work contains some of the ideas derived from conversations of the author and Hojoo Lee, and also some of them are consequence of their collaboration. The author would like to express his gratitude to Hojoo for taking him to the world of divergence and duality, specially during his stays at KIAS in 2013 and 2014. This research was partially supported by Spanish MEC-Feder Research Project MTM2014-52368-P, and by the EPSRC Grant No. EP/M024512/1. }



\begin{document}
\maketitle

\begin{abstract}
The main aim of this survey paper is to gather together some results concerning the Calabi type duality discovered by Lee~\cite{Lee} between certain families of spacelike graphs with constant mean curvature in Riemannian and Lorentzian homogeneous 3-manifolds with isometry group of dimension 4. The duality is conformal and swaps mean curvature and bundle curvature, and we will revisit it by giving a more general statement in terms of conformal immersions. This will show that many features in the theory of surfaces with mean curvature $\frac{1}{2}$ in $\mathbb{H}^2\times\mathbb{R}$ or minimal surfaces in the Heisenberg space have nice geometric interpretations in their dual Lorentzian counterparts. We will briefly discuss some applications such as gradient estimates for entire minimal graphs in Heisenberg space~\cite{MN} and the existence of complete spacelike surfaces~\cite{LeeMan}, and we will give an uniform treatment to the behavior of the duality with respect to ambient isometries. Finally, some open questions are posed in the last section.
\end{abstract}

\section{Introduction} 

The theory of constant mean curvature surfaces in homogeneous 3-manifolds has become relatively relevant during the last decades. Although all started in the ambient spaces with constant sectional curvature $\mathbb{R}^3$, $\mathbb{S}^3$ and $\mathbb{H}^3$, then product spaces $\mathbb{H}^2\times\mathbb{R}$ and $\mathbb{S}^2\times\mathbb{R}$ were also considered, and soon afterwards the Heisenberg space $\mathrm{Nil}_3$, the Berger spheres $\mathbb{S}^3_{\text{Berger}}$, and $\widetilde{\mathrm{Sl}}_2(\mathbb{R})$ (the universal cover of the special linear group equipped with some distinguished left-invariant metrics) came into scene. All these 3-manifolds, except for $\mathbb{H}^3$, lie in a 2-parameter family $\mathbb{E}(\kappa,\tau)$ with the property that they admit a Killing submersion with constant bundle curvature $\tau$ over the simply connected surface with constant curvature $\kappa$ (see definitions in Section~\ref{sec:ekt}). Abresch and Rosenberg's discovery of a geometric quadratic differential which is holomorphic on constant mean curvature surfaces~\cite{AR}, together with the Lawson type correspondence found by Daniel~\cite{Dan}, and the solution to the Bernstein problem in Heisenberg space given by Fern\'{a}ndez and Mira~\cite{FM} encouraged many geometers to work on this theory. We refer the reader to the introductory lecture notes written by Daniel, Hauswirth and Mira~\cite{DHM}.

It is worth pointing out that the spaces $\mathbb{E}(\kappa,\tau)$, $\kappa-4\tau^2\neq 0$, model all simply connected 3-manifolds with isometry group of dimension 4, but recently an increasing attention has been paid to the family of all simply connected homogeneous 3-manifolds, consisting essentially of Lie groups with arbitrary left-invariant metrics. A sample of this is the recent classification of spheres with constant mean curvature in such Lie groups by Meeks, Mira, P\'{e}rez and Ros~\cite{MMPR}. It would be interesting to know whether or not the ideas we will discuss below can be extended to the Lie-group setting. 

The spaces $\mathbb{E}(\kappa,\tau)$ have Lorentzian counterparts $\mathbb{L}(\kappa,\tau)$ also admitting a Killing-submersion structure over $\mathbb{M}^2(\kappa)$ with constant bundle curvature $\tau$, but in such a way that the Killing direction is timelike whereas the horizontal distribution is spacelike. Calabi stated in~\cite{Calabi} an interesting duality between minimal graphs in Euclidean space $\mathbb{R}^3=\mathbb{E}(0,0)$ and maximal graphs in Lorentz-Minkowski space $\mathbb{L}^3=\mathbb{L}(0,0)$, based on a clever trick using Poincar\'e's Lemma. Although this duality is often called Calabi duality, it seems that similar ideas had been already considered by other authors, see for instance the work of Catalan~\cite{Catalan} more than one century before. Calabi's duality has a natural interpretation in terms of null curves and holomorphic Weierstrass data (e.g., see~\cite{LLS} where the duality has been used to study conical singularities of maximal surfaces in $\mathbb{L}^3$).

Later on, Calabi's work was generalized to a correspondence between minimal graphs in $\mathbb{S}^2\times\mathbb{R}=\mathbb{E}(1,0)$ and maximal graphs in $\mathbb{S}^2\times\mathbb{R}_1=\mathbb{L}(1,0)$ by Albujer and Al\'{i}as~\cite{AA}, and to a correspondence between constant mean curvature $H$ graphs in $\mathbb{E}(\kappa,\tau)$ and spacelike graphs with constant mean curvature $\tau$ in $\mathbb{L}(\kappa,H)$ by Lee~\cite{Lee}. In Theorem~\ref{thm:duality} we will rewrite Lee's result in the more general language of conformal immersions by allowing the surfaces to be multigraphs rather than graphs. In order to get to this result, we will employ the classification of Killing submersions in~\cite{Man}, though the same statement can also be achieved by means of analytic continuation. The rest of Section~\ref{sec:duality} will be devoted to review other well-known properties of the duality, as well as to recover its explicit expression in coordinates as a first-order \textsc{pde} system, where one can see that the mean curvature equations become the integrability conditions of the system. Although we will give some examples, more of them can be found in~\cite{KL,Lee}.

In Section 3, we will revisit some results in~\cite{LeeMan} showing that the space $\mathbb{L}(\kappa,\tau)$ does not admit complete spacelike surfaces if $\kappa+4\tau^2>0$, and we will provide three proofs of this result of different nature. We will also consider \emph{dual-complete} surfaces rather than complete ones, which is nothing but assuming that the dual surface is complete in $\mathbb{E}(\kappa,H)$. We will use well-known results in the Riemannian setting to give a fairly complete classification of dual-complete spacelike surfaces in $\mathbb{L}(\kappa,\tau)$, see also~\cite{DHM,MR}. This will grant the opportunity to connect the duality with the theory of constant mean curvature surfaces in $\mathbb{E}(\kappa,\tau)$, and we will discuss the relation between the duality and Daniel correspondence~\cite{Dan}, see Remark~\ref{rmk:daniel}.  

Critical mean curvature surfaces in $\mathbb{E}(\kappa,\tau)$-spaces are those whose dual surfaces live in constant sectional curvature spaces (i.e., in the Lorentz-Minkowski space $\mathbb{L}^3=\mathbb{L}(0,0)$ or in the anti-de Sitter space $\mathbb{H}_1^3(\kappa)=\mathbb{L}(\kappa,\frac{1}{2}\sqrt{-\kappa})$), that is to say those whose mean curvature $H$ satisfies the relation $4H^2+\kappa=0$. This gives an outstanding interpretation of why surfaces with critical mean curvature have features that the rest of constant mean curvature surfaces fail to enjoy:
\begin{itemize} 
	\item The hyperbolic Gauss map for mean curvature $\frac{1}{2}$ surfaces in $\mathbb{H}^2\times\mathbb{R}$ discovered by Fern\'{a}ndez and Mira~\cite{FM}), the left-invariant Gauss map for minimal surfaces in $\mathrm{Nil}_3(\tau)$ given by Daniel~\cite{Dan2}, and the rest of harmonic Gauss maps for critical mean curvature surfaces, see~\cite{DFM}, correspond to the classical (hyperbolic) Gauss map of surfaces in $\mathbb{L}^3$ and $\mathbb{H}_1^3(\kappa)$, see also~\cite{Lee2}.
	\item The Abresch-Rosenberg differential of surfaces with critical mean curvature, which is associated with the harmonic Gauss maps, corresponds to the classical Hopf differential in $\mathbb{L}^3$ and $\mathbb{H}_1^3(\kappa)$.
\end{itemize}

In section 4, we will discuss an application in the other direction by employing a Lorentzian property (an estimate for entire graphs in $\mathbb{L}^3$ proved by Cheng and Yau) to get some gradient estimates for entire minimal graphs in $\mathrm{Nil}_3(\tau)$. This result is part of~\cite{MN}, though we will use it here to prove also that entire graphs with positive mean curvature in $\mathbb{L}^3$ have infinite area. 

Some other properties will be analized in Section 5. First, we will give a general new result for understanding the behaviour of the duality with respect to isometries, generalizing~\cite{Lee}. Basically we will prove that direct isometries preserving the unit Killing vector field behave well with respect to the duality, in the sense that they correspond to other isometries in the dual setting also preserving the Killing vector field. Such isometries represent a 4-parameter subgroup of isometries of $\mathbb{E}(\kappa,\tau)$ or $\mathbb{L}(\kappa,\tau)$, which is generically the total dimension of the isometry group. Nonetheless, in the cases the correspondence involves an space form, there are 2 extra dimensions in the isometry. Again this applies to the case of critical mean curvature, providing geometric 2-parameter deformations (not by ambient isometries) in the following families of surfaces:
\begin{enumerate}
	\item Surfaces in $\mathbb{E}(\kappa,\tau)$ with critical constant mean curvature $H$, i.e., such that $\kappa+4H^2=0$.
	\item Surfaces in $\mathbb{L}(\kappa,\tau)$ with critical constant mean curvature $H$, i.e., such that $\kappa-4H^2=0$.
\end{enumerate}
Note that the change of sign in the conditions $\kappa+4H^2=0$ and $\kappa-4H^2=0$ is not a mistake, and comes from the fact that space forms satisfy $\kappa-4\tau^2=0$ in the Riemannian case, and $\kappa+4\tau^2=0$ in the Lorentzian case (see Section~\ref{sec:ekt} and the comments below Proposition~\ref{prop:transform}).

In the last section, we will leave three open original questions whose solutions could improve the understanding of the correspondence in the opinion of the author. We also refer to~\cite{DHM} for a list of open conjectures in $\mathbb{E}(\kappa,\tau)$-spaces, though some of them have been already solved since the document was published in 2009.

\section{A conformal duality}\label{sec:duality}

\subsection{The spaces $\mathbb{E}(\kappa,\tau)$ and $\mathbb{L}(\kappa,\tau)$}\label{sec:ekt}

Let $\tau\in C^\infty(M)$ be a smooth function on a simply connected orientable Riemannian surface $M$. From~\cite{Man}, we get that there is a unique Riemannian submersion $\pi:\mathbb{E}\to M$ such that
\begin{enumerate}
 \item $\mathbb{E}$ is a simply connected orientable Riemannian 3-manifold,
 \item the fibers of $\pi$ are the integral curves of a unit Killing vector field $\xi\in\mathfrak{X}(\mathbb{E})$, and
 \item $\overline\nabla_X\xi=(\tau\circ\pi) X\times\xi$, for all $X\in\mathfrak{X}(\mathbb{E})$, where $\overline\nabla$ denotes the Levi-Civita connection in $\mathbb{E}$ and $\times$ stands for the cross product.
\end{enumerate} 
Such a Riemannian submersion is called a \emph{unit Killing submersion} with \emph{bundle curvature} $\tau$ over $M$ (this result also holds in the more general non-simply connected and non-unitary case, see~\cite{LerMan}). Similar arguments to those in~\cite{Man} lead to the existence and uniqueness of a Riemannian submersion $\pi:\mathbb{L}\to M$ fulfilling items 1--3 above, and such that $\mathbb{L}$ is Lorentzian and $\xi\in\mathfrak{X}(\mathbb{L})$ is timelike. 

Let us assume that $\pi:\mathbb{E}\to M$ is a (Riemannian or Lorentzian) Killing submersion over $M$ with bundle curvature $\tau$, and whose fibers have infinite length. Given an open set $\Omega\subset M$ and a smooth section $F_0:\Omega\to\mathbb{E}$ (the existence of $F_0$ is guaranteed since the fibers have infinite length, see~\cite{LerMan,Man}), any section of $\pi$ over $\Omega$ can be parameterized in terms of a function $u:\Omega\to\mathbb{R}$ as
\begin{equation}\label{eqn:grafo}
F_u:\Omega\to\mathbb{E},\qquad F_u(p)=\phi_{u(p)}(F_0(p)),
\end{equation}
where $\{\phi_t:t\in\mathbb{R}\}$ is the 1-parameter group of isometries associated with $\xi$, the so-called vertical translations. Such a section $F_u$ is usually called a Killing (or vertical) graph over $\Omega$.

Let $d$ be the signed Killing distance to $F_0$ determined unambiguously by the relation $\phi_{d(p)}(F_0(p))=p$ for all $p\in\mathbb{E}$, and consider the projection of its gradient $Z=\pi_*(\overline\nabla d)\in\mathfrak{X}(M)$. Given $u\in C^2(\Omega)$, we define the \emph{generalized gradient} of $u$ as the vector field 
\begin{equation}\label{eqn:G}
Gu=\nabla u-\epsilon Z\in\mathfrak{X}(M),
\end{equation}
where we will take $\epsilon=1$ or $\epsilon=-1$ depending on whether $\mathbb{E}$ is Riemannian or Lorentzian, respectively. It is easy to check that $F_u$ defines a spacelike surface (i.e., the induced metric in $F_u$ is positive definite) if and only if 
\begin{equation}\label{eqn:spacelike}
1+\epsilon\|Gu\|^2>0,
\end{equation} 
in which case the mean curvature of $F_u$ can be computed (as a function on $\Omega$) as
\begin{equation}\label{eqn:H}
H=\frac{1}{2}\mathrm{div}\left(\frac{Gu}{\sqrt{1+\epsilon\|Gu\|^2}}\right),
\end{equation}
where the divergence, gradient and norm are computed in the geometry of $M$. Although there is no explicit dependence upon the bundle curvature in~\eqref{eqn:G} or~\eqref{eqn:H}, it is encoded in the vector field $Z$ in the sense that
\begin{equation}\label{eqn:tau}
\mathrm{div}(JZ)=-2\tau,
\end{equation}
where $J$ is the $\frac{\pi}{2}$-rotation in the tangent bundle $TM$, see~\cite[Lemma~3.2]{LerMan}. In the duality below, we will essentially exploit the fact that both $H$ and $\tau$ admit the divergence type expressions~\eqref{eqn:H} and~\eqref{eqn:tau}.

Killing submersions give a good framework for studying simply connected homogeneous Riemannian or Lorentzian 3-manifolds with isometry group of dimension 4. These 3-manifolds are classified in two families $\mathbb{E}(\kappa,\tau)$ and $\mathbb{L}(\kappa,\tau)$, where $\mathbb{E}(\kappa,\tau)$ (resp. $\mathbb{L}(\kappa,\tau)$) stands for the total space of the unique Riemannian (resp. Lorentzian) Killing submersion with constant bundle curvature $\tau$ over $\mathbb{M}^2(\kappa)$, the simply connected surface of constant curvature $\kappa$. For each particular choice of $(\kappa,\tau)$ we get the following spaces (we use an index 1 to distinguish the Lorentzian case):
\begin{center}
 \begin{tabular}{c||c|c|c}
 &$\kappa<0$&$\kappa=0$&$\kappa>0$\\\hline\hline
 $\tau=0$&$\mathbb{H}^2(\kappa)\times\mathbb{R}$&$\mathbb{R}^3$&$\mathbb{S}^2(\kappa)\times\mathbb{R}$\\
 &$\mathbb{H}^2(\kappa)\times\mathbb{R}_1$&$\mathbb{L}^3$&$\mathbb{S}^2(\kappa)\times\mathbb{R}_1$\\\hline
 $\tau\neq 0$&$\widetilde{\mathrm{Sl}}_2(\mathbb{R})(\kappa,\tau)$&$\mathrm{Nil}_3(\tau)$&$\mathbb{S}^3_{\text{Berger}}(\kappa,\tau)$\\
 &$\widetilde{\mathrm{Sl}}^1_2(\mathbb{R})(\kappa,\tau)$&$\mathrm{Nil}^1_3(\tau)$&$\mathbb{S}^{3,1}_{\text{Berger}}(\kappa,\tau)$
 \end{tabular}\end{center} 
There are some space forms hidden in this table, whose isometry group has dimension 6, namely the Riemannian $\mathbb{R}^3$ and $\mathbb{S}^3(\kappa)$ show up in the family $\mathbb{E}(\kappa,\tau)$ with $\kappa-4\tau^2=0$, whereas the Lorentz-Minkowski space $\mathbb{L}^3$ and the anti-de Sitter space $\mathbb{H}^3_1(\kappa)$ are $\mathbb{L}(\kappa,\tau)$-spaces with $\kappa+4\tau^2=0$. We remark that the hyperbolic space $\mathbb{H}^3(\kappa)$ or the de-Sitter space $\mathbb{S}^3_1(\kappa)$ do not admit unit Killing vector fields.

\subsection{An extended conformal duality}

In the previous section, a graph in a Riemannian or Lorentzian Killing submersion $\pi:\mathbb{E}\to M$ has been defined as a smooth section of $\pi$, i.e., a surface $\Sigma$ in $\mathbb{E}$ such that $\pi_{|\Sigma}:\Sigma\to M$ is a diffeomorphism onto its image, and in particular a graph is transversal to the vertical Killing vector field $\xi$. The graph is said entire if $\pi_{|\Sigma}$ is also surjective. In the Riemannian case, these conditions can be weakened to any of the following three equivalent conditions:
\begin{enumerate}
	\item[i.] $\pi_{|\Sigma}$ is a local diffeomorphism,
	\item[ii.] $\Sigma$ is nowhere vertical, i.e., everywhere transversal to $\xi$,
	\item[iii.] the angle function $\nu=\langle N,\xi\rangle$ has no zeroes, where $N$ is a unit normal to $\Sigma$. 
\end{enumerate}
If any of these properties hold, the surface $\Sigma$ is usually called a vertical (spacelike) \emph{multigraph}. In the Lorentzian case, spacelike surfaces are automatically transversal to the Killing direction, so multigraphs turn out to be natural objects in the Lorentzian setting.

As pointed out in the introduction our version of the duality assumes that the surfaces are multigraphs rather than graphs, which generalizes the strictly graphical duality given by Lee~\cite{Lee} in the case $\kappa\leq 0$, and by Albujer and Al\'{i}as~\cite{AA} in the case of $\mathbb{S}^2\times\mathbb{R}$. Although we will assume simple connectedness in the statement, it can be applied to any conformal immersion by considering the universal cover of the Riemann surface we are working with.

\begin{theorem}\label{thm:duality}
Let $\Sigma$ be a simply connected Riemann surface, and let $\kappa,\tau,H\in\mathbb{R}$. There is a correspondence between
\begin{enumerate}
	\item[(a)] Conformal immersions $X:\Sigma\to\mathbb{E}(\kappa,\tau)$ with constant mean curvature $H$ and nowhere vertical.
	\item[(b)] Conformal spacelike immersions $\widetilde X:\Sigma\to\mathbb{L}(\kappa,H)$ with constant mean curvature $\tau$. 
\end{enumerate}
The corresponding immersions $X$ and $\widetilde X$ are determined up to a vertical translation and satisfy $\pi\circ X=\widetilde\pi\circ\widetilde X$.
\end{theorem}

\begin{proof}
Let $X:\Sigma\to\mathbb{E}(\kappa,\tau)$ be a conformal immersion with constant mean curvature $H$. Since $X$ is nowhere vertical, the map $\pi\circ X$ is a local diffeomorphism and we can consider the Riemannian surface $M$ defined as $\Sigma$ endowed with the pullback of the metric of $\mathbb{M}^2(\kappa)$ by $\pi\circ X$. As $M$ is simply connected, there are unique Killing submersions $\pi':\mathbb{E}\to M$ and $\widetilde\pi':\mathbb{L}\to M$, where $\pi'$ is Riemannian and $\widetilde\pi'$ is lorentzian) with constant bundle curvatures $\tau$ and $H$, respectively, and such that $\mathbb{E}$ and $\mathbb{L}$ are simply connected.

Note that $\mathbb{E}$ and $\mathbb{L}$ are locally isometric to $\mathbb{E}(\kappa,\tau)$ and $\mathbb{L}(\kappa,H)$, respectively. More explicitly, since $\pi\circ X:M\to\mathbb{M}^2(\kappa)$ is a local isometry, then it lifts to local isometries $\Psi:\mathbb{E}\to\mathbb{E}(\kappa,\tau)$ and $\widetilde\Psi:\mathbb{L}\to\mathbb{L}(\kappa,H)$ such that the following diagram commutes:
\[\begin{tikzcd}
\mathbb{E}\arrow[rr,rightarrow, "\pi'"]&&
M\arrow[rr,leftarrow,"\widetilde\pi'"]&&
\mathbb{L}\arrow[d,rightarrow, "\widetilde\Psi"]\\
\mathbb{E}(\kappa,\tau)\arrow[rr, rightarrow, "\pi"]\arrow[u,leftarrow, "\Psi"]&&
\mathbb{M}^2(\kappa)\arrow[u, leftarrow, "\pi\circ X"]\arrow[rr,leftarrow,"\widetilde\pi"]&&
\mathbb{L}(\kappa,H)
\end{tikzcd}\]
Furthermore, the image of $X$ lies in the image of $\Psi$, so the immersion $\Psi^{-1}\circ X:\Sigma\to\mathbb{E}$ is well defined, and it is not only a conformal immersion in $\mathbb{E}$ with constant mean curvature $\tau$ but also an entire graph over $M$. Note that fibers of $\pi'$ have infinite length, since otherwise $\pi'$ would be the Hopf fibration, which does not admit global sections (this excludes the case $\Sigma$ is a sphere and $\tau\neq 0$, see the comments below).

Let $F_0:M\to\mathbb{E}$ be a global section of $\pi'$, which exists by~\cite[Proposition~3.3]{LerMan}. Then there is a function $u$ such that $\Psi^{-1}\circ X$ can be parameterized as $F_u:M\to\mathbb{E}$ where $F_u(p)=\phi_{u(p)}(F_0(p))$ and $\{\phi_t:t\in\mathbb{R}\}$ stands for the 1-parameter group of isometries associated to the vertical Killing vector field in $\mathbb{E}$. The fact that $F_u$ has constant mean curvature $H$ can be expressed as
\begin{equation}\label{thm:duality:eqn1}
\mathrm{div}\left(\frac{Gu}{\sqrt{1+\|Gu\|^2}}\right)=2H,
\end{equation}
where $Gu=\nabla u-Z$, and $Z=\pi'_*(\overline\nabla d)\in\mathfrak{X}(M)$ satisfies $\mathrm{div}(JZ)=-2\tau$.

Let $\widetilde\pi':\mathbb{L}\to M$ be the unique Lorentzian Killing submersion over $M$ with constant bundle curvature $H$ such that $\mathbb{L}$ is simply connected. Given an initial global section $\widetilde F_0:M\to\mathbb{L}$ and the signed distance $\widetilde d$, the vector field $\widetilde Z= \widetilde \pi'_*(\overline\nabla\widetilde d)\in\mathfrak{X}(M)$ satisfies $\mathrm{div}(J\widetilde Z)=-2H$. Hence we can rewrite~\eqref{thm:duality:eqn1} as  
\begin{equation}\label{thm:duality:eqn2}
\mathrm{div}\left(\frac{Gu}{\sqrt{1+\|Gu\|^2}}+J\widetilde Z\right)=0,
\end{equation}
The fact that $M$ is simply connected together with Poincar\'e's Lemma tell us that the term inside the divergence in~\eqref{thm:duality:eqn2} equals $-J\nabla v$ for some $v\in C^\infty(M)$. Setting $\widetilde G v=\nabla v+\widetilde Z$, we reach 
\begin{equation}\label{thm:duality:eqn3}
\frac{Gu}{\sqrt{1+\|Gu\|^2}}=-J\widetilde G v\quad\Longleftrightarrow\quad
\frac{\widetilde Gv}{\sqrt{1-\|\widetilde Gv\|^2}}=JGu
\end{equation}
Observe that, in order to check the equivalence in~\eqref{thm:duality:eqn3}, we can take squared norms in the first expression to obtain $\|\widetilde Gv\|<1$ (which implies that the entire graph $\widetilde F_v:M\to\mathbb{L}$ given by $\widetilde F_v(p)=\widetilde\phi_{v(p)}(\widetilde F_0(p))$ is spacelike, see~\eqref{eqn:spacelike}), and the identity $1+\|Gu\|^2=(1-\|\widetilde Gv\|^2)^{-1}$. From here it is easy to prove~\eqref{thm:duality:eqn3}. 

Applying now the divergence operator, we get that 
\begin{equation}\label{thm:duality:eqn5}
\mathrm{div}\left(\frac{\widetilde Gv}{\sqrt{1-\|\widetilde Gv\|^2}}\right)=\mathrm{div}(JGu)=\mathrm{div}(J\nabla u)-\mathrm{div}(JZ)=2\tau,
\end{equation}
i.e., $\widetilde F_v$ has  constant mean curvature $\tau$ in $\mathbb{L}$.

The diffeomorphism $T:F_u(M)\to\widetilde F_v(M)$ determined by the relation $\widetilde\pi'\circ T=\pi'$ is conformal (the proof of this property will be postponed to Section~\ref{sec:coordinates} since it follows from the same argument as in~\cite{Lee,LeeMan} in local coordinates). Hence, it suffices to define $\widetilde X=\widetilde\Psi\circ T\circ\Psi^{-1}\circ X$, which is conformal as the composition of conformal maps, and trivially satisfies $\pi\circ X=\widetilde\pi\circ\widetilde X$.

The way back from a constant mean curvature $\tau$ spacelike graph in $\mathbb{L}(\kappa,H)$ is a consequence of a completely analogous argument, so it will be skipped.
\end{proof}

\begin{remark}~
\begin{enumerate}
	\item The conformal immersion $X$ is a graph if and only if $\widetilde X$ is a graph, in which case the domains of both graphs in $\mathbb{M}^2(\kappa)$ coincide. This case will be analyzed in detail in Section~\ref{sec:coordinates}.
	\item The duality is determined unambiguously by~\eqref{thm:duality:eqn3} up to a vertical translation, which implies it is one-to-one up to vertical translations. The analytic reason is that, when Poincar\'e's Lemma is applied, the function we introduce is determined up to an additive constant. 
	\item In the case $\Sigma$ is a sphere, since the angle functions of the immersions are positive continuous functions, they must be bounded between two positive constants. Therefore $\pi$ restricts to the surface as a covering map, which has to be one-to one because $\Sigma$ is simply connected. We deduce that $\mathbb{M}^2(\kappa)$ must be also a sphere, i.e., $\kappa>0$. The bundle curvature must be zero, for were it not the case, the Killing submersion would be topologically the Hopf fibration, which does not admit global sections and we would get a contradiction. Moreover, integrating the divergence equation for the mean curvature we get that $H=0$, so this case only leads to horizontal slices $\mathbb{S}^2(\kappa)\times\{t_0\}$, which are dual in $\mathbb{S}^2(\kappa)\times\mathbb{R}$ and $\mathbb{S}^2(\kappa)\times\mathbb{R}_1$.
	\item In~\cite{LeeMan} a much more general graphical duality is obtained, where $H$ and $\tau$ are assumed arbitrary over an arbitrary simply connected non-compact base surface. Nonetheless, there is no direct extension of Theorem~\ref{thm:duality} to this general setting, due to the fact that the proof of Theorem~\ref{thm:duality} is implicitly assuming that the multigraph has the same mean curvature at points lying in the same vertical fiber, which is not a natural assumption in the prescribed mean curvature scenario.
\end{enumerate}
\end{remark}

\subsection{The duality in coordinates}\label{sec:coordinates}

We will apply now Theorem~\ref{thm:duality} to the case the surfaces are graphs, introducing the standard models of $\mathbb{E}(\kappa,\tau)$ and $\mathbb{L}(\kappa,\tau)$ in order to recover the explicit equations defining the duality given by Lee~\cite{Lee}. Let $\kappa\in\mathbb{R}$, and consider the function
\[\lambda_\kappa(x,y)=\left(1+\tfrac{\kappa}{4}(x^2+y^2)\right)^{-1},\]
defined on the open set $\Omega_\kappa=\{(x,y)\in\mathbb{R}^2:1+\frac{\kappa}{4}(x^2+y^2)>0\}$, which is a disk for $\kappa<0$ and the whole plane otherwise. The conformal metric $\lambda_k^2(\mathrm{d} x^2+\mathrm{d} y^2)$ has constant curvature $\kappa$, and it is (locally) isometric to $\mathbb{M}^2(\kappa)$. We get the well known rotationally invariant models 
\begin{equation}\label{eqn:models}
\begin{aligned}
\mathbb{E}(\kappa,\tau)&=\left(\Omega_\kappa\times\mathbb{R},\lambda_k^2(\mathrm{d} x^2+\mathrm{d} y^2)+(\mathrm{d} z+\tau\lambda_k(y\,\mathrm{d} x-x\,\mathrm{d} y))^2\right),\\
\mathbb{L}(\kappa,\tau)&=\left(\Omega_\kappa\times\mathbb{R},\lambda_k^2(\mathrm{d} x^2+\mathrm{d} y^2)-(\mathrm{d} z-\tau\lambda_k(y\,\mathrm{d} x-x\,\mathrm{d} y))^2\right),
\end{aligned}
\end{equation}
on which the Killing submersion takes the form $(x,y,z)\mapsto (x,y)$, and $\partial_z$ is the unit vertical Killing vector field. It is necessary to point out that this model omits a point in the base $\mathbb{M}^2(\kappa)$ when $\kappa>0$, and hence a whole fiber in $\mathbb{E}(\kappa,\tau)$ or $\mathbb{L}(\kappa,\tau)$. In the proof of Theorem~\ref{thm:duality} this issue did not appear since we were working from a coordinate-free point of view.

The global section $F_0(x,y)=(x,y,0)$ allows us to parameterize the graph $F_u$ of a function $u\in C^\infty(\Omega)$ on an open subset $\Omega\subset\Omega_\kappa$ as usual:
\[F_u(x,y)=(x,y,u(x,y)).\]
If $\Omega$ is open and simply connected, and we consider dual graphs $F_u$ in $\mathbb{E}(\kappa,\tau)$ and $F_v$ in $\mathbb{L}(\kappa,H)$, the generalized gradients $Gu$ of $F_u$ and $\widetilde Gv$ of $F_v$ can be expressed in coordinates as
\begin{align*}
Gu&=\alpha\frac{\partial_x}{\lambda_\kappa}+\beta\frac{\partial_y}{\lambda_\kappa},\quad\text{where }
\begin{cases}\alpha=\frac{u_x}{\lambda_\kappa}+\tau y,\\\beta=\frac{u_y}{\lambda_\kappa}-\tau x,\end{cases}\\
\widetilde Gv&=\widetilde\alpha\frac{\partial_x}{\lambda_\kappa}+\widetilde\beta\frac{\partial_y}{\lambda_\kappa},\quad\text{where }
\begin{cases}\widetilde\alpha=\frac{v_x}{\lambda_\kappa}-\tau y,\\\widetilde\beta=\frac{v_y}{\lambda_\kappa}+\tau x.\end{cases}
\end{align*}
In order to simplify the notation, in the sequel we shall also consider 
\begin{align*}
\omega&=\sqrt{1+\alpha^2+\beta^2},&\widetilde\omega&=\sqrt{1-\widetilde\alpha^2-\widetilde\beta^2}.
\end{align*}
Then the first identity in~\eqref{thm:duality:eqn3} yields an explicit \textsc{pde} system which allows us to solve for $u$ in terms of $v$:
\begin{align*}
 \alpha=\frac{\widetilde\beta}{\widetilde\omega}\quad\Leftrightarrow\quad u_x&=\frac{v_y+Hx\lambda_k}{\sqrt{1+(\frac{v_x}{\lambda_k}-H y)^2+(\frac{v_y}{\lambda_k}+H x)^2}}-\tau y\lambda_k,\\
 \beta=\frac{-\widetilde\alpha}{\widetilde\omega}\quad\Leftrightarrow\quad u_y&=\frac{-v_x+Hy\lambda_k}{\sqrt{1+(\frac{v_x}{\lambda_k}-H y)^2+(\frac{v_y}{\lambda_k}+H x)^2}}+\tau x\lambda_k.
 \end{align*}
 Likewise, the second identity in~\eqref{thm:duality:eqn3} gives the \textsc{pde} system:
\begin{align*}
 \widetilde\alpha=\frac{-\beta}{\omega}\quad\Leftrightarrow\quad v_x&=\frac{-u_y+\tau x\lambda_k}{\sqrt{1+(\frac{u_x}{\lambda_k}+\tau y)^2+(\frac{u_y}{\lambda_k}-\tau x)^2}}+Hy\lambda_k,\\
 \widetilde\beta=\frac{\alpha}{\omega}\quad\Leftrightarrow\quad v_y&=\frac{u_x+\tau y\lambda_k}{\sqrt{1+(\frac{u_x}{\lambda_k}+\tau y)^2+(\frac{u_y}{\lambda_k}-\tau x)^2}}-Hx\lambda_k.
 \end{align*}
These four equations are the so-called \emph{twin relations}, and allow us to solve for $u$ (resp. $v$) when $v$ (resp. $u$) is known. The constant mean curvature equations for $F_u$ and $F_v$ can be regarded as the integrability conditions for the twin relations. Also from the twin relations we get the quite useful formula
\begin{equation}\label{eqn:omega}
\omega\cdot\widetilde\omega=1.
\end{equation}

On the other hand, we can define global references $\{e_1,e_2\}$ tangent to $F_u$ and $\{\widetilde e_1,\widetilde e_2\}$ tangent to $F_v$ (not necessarily orthogonal) as
\begin{equation}\label{eqn:tangent-frame}
\begin{aligned}
 e_1&=\frac{\partial_x+u_x\partial_z}{\lambda_\kappa},&\widetilde e_1&=\frac{\partial_y+v_y\partial_z}{\lambda_\kappa},\\
 e_2&=\frac{\partial_x+v_x\partial_z}{\lambda_\kappa},&\widetilde e_2&=\frac{\partial_y+v_y\partial_z}{\lambda_\kappa}.
\end{aligned}\end{equation}
Using~\eqref{eqn:models} it is easy to check that the first fundamental forms $I$ of $F_u$ and $\widetilde I$ of $F_v$ can be expressed in these references as the matrices
\begin{equation}\label{eqn:1ff}
\begin{aligned}
 \mathrm{I}&\equiv\left(\begin{matrix}
 1+\alpha^2&\alpha\beta\\\alpha\beta&1+\beta^2
 \end{matrix}\right),&
 \widetilde{\mathrm{I}}&\equiv\left(\begin{matrix}
 1-\widetilde\alpha^2&-\widetilde\alpha\widetilde\beta\\-\widetilde\alpha\widetilde\beta&1-\widetilde\beta^2
 \end{matrix}\right).
\end{aligned}\end{equation}
From the twin relations we infer that the two matrices in Equation~\eqref{eqn:1ff} satisfy $\widetilde{\mathrm{I}}=\omega^{-2}\mathrm{I}$. Since the global diffeomorphism $T:F_u(\Omega)\to F_v(\Omega)$ given in the proof of Theorem~\ref{thm:duality} reads $T(x,y,u(x,y))=(x,y,v(x,y))$ for all $(x,y)\in\Omega$, it follows that $T_*e_i=\widetilde e_i$ for $i\in\{1,2\}$, so $T$ is conformal (this was was the missing bit in the proof of Theorem~\ref{thm:duality}).

Let $N$ and $\widetilde N$ be upward pointing unit vectors to $F_u$ and $F_v$, respectively. Their horizontal parts satisfy the following identities (see~\cite[Equation~3.3]{LerMan}):
\begin{equation}\label{eqn:N}
\begin{aligned}
\pi_*N&=\frac{Gu}{\sqrt{1+|Gu|^2}},&
\widetilde\pi_*\widetilde N=\frac{\widetilde Gv}{\sqrt{1-|\widetilde Gv|^2}}.
\end{aligned}
\end{equation}
Since $N$ and $\widetilde N$ are unitary, their vertical components (also called \emph{angle functions} of the immersions $X$ and $\widetilde X$, respectively), are given by
\begin{equation}\label{eqn:angle}
\begin{aligned}
 \nu=\langle N,\partial_z\rangle&=\frac{1}{\sqrt{1+|Gu|^2}}=\frac{1}{\omega},\\
 \widetilde\nu=\langle \widetilde N,\partial_z\rangle&=\frac{1}{\sqrt{1-|\widetilde Gv|^2}}=\frac{1}{\widetilde\omega}.
\end{aligned}
\end{equation}
From~\eqref{eqn:omega} and~\eqref{eqn:angle}, we deduce that dual conformal immersions have reciprocal angle functions. Note that the angle function $\nu$ (resp. $\widetilde\nu$) is the cosine (resp. hyperbolic cosine) of the angle between the upward-pointing normal and the Killing vector field, so it satisfies $0<\nu\leq 1$ (resp. $1\leq\widetilde\nu<\infty$).

As a conclusion to this section, we will give several examples, though we will not include the computations, which should be easy to deduce from the twin relations. Further examples can be found in~\cite{KL,Lee}.

\begin{example}\label{example1}
Let us consider the function $v(x,y)=0$, which defines a zero mean curvature graph in $\mathbb{L}(\kappa,H)$ but fails to be spacelike over the whole base surface $\mathbb{M}^2(\kappa)$ when $\kappa+4 H^2>0$. The spacelike condition only holds in a disk centered at $(0,0)\in\mathbb{M}^2(\kappa)$. The dual surface $F_u$ is half of the rotationally invariant sphere with constant mean curvature $H$ in $\mathbb{E}(\kappa,\tau)$.

The value of $H$ (if it exists) such that $\kappa+4 H^2=0$ is called the critical mean curvature in $\mathbb{E}(\kappa,\tau)$. Note that constant mean curvature spheres only exist when the mean curvature is supercritical (i.e., $\kappa+4 H^2>0$) and the duality provides a nice way to obtain explicit expressions for them.
\end{example}

\begin{example}\label{example2}
Let us consider a constant mean curvature $H$ graph in $\mathbb{R}^3$ which is rotationally invariant about the $x$-axis, i.e, a graph of a function $u$ of the form $u(x,y)=\sqrt{r(x)^2-y^2}$ for some $r\in C^\infty(I)$ on an interval $I\subseteq\mathbb{R}$ (a piece of a Delaunay surface). The dual maximal surface $F_v$ in $\mathrm{Nil}_3(H)$ is of the form $v(x,y)=yf(x)$ for some function $f\in C^\infty(I)$ (in particular, $F_v$ is a surface ruled by Euclidean lines). Quite surprisingly, as in the above case of spheres, $v(x,y)$ still defines a graph with zero mean curvature outside the domain of $u$.
\end{example}

\begin{example}\label{example2}
Let us consider now a constant mean curvature $H$ spacelike graph in $\mathbb{L}^3$ which is invariant under hyperbolic rotations about the $x$-axis, i.e, a graph of a function $v$ of the form $v(x,y)=\sqrt{r(x)^2+y^2}$ for some $r\in C^\infty(I)$ on an interval $I\subseteq\mathbb{R}$. These examples were studied by Daniel in~\cite{Dan2}, and include the surfaces parameterized by
\begin{align*}
 (x,y)&\mapsto \left(x,y,\sqrt{H^{-2}+x^2+y^2}\right)\\
 (x,y)&\mapsto \left(x,y,\sqrt{(2H)^{-2}+y^2}\right)\\
 (x,y)&\mapsto \tfrac{1}{H}\left(x-\tfrac{1}{2}\coth(x),\tfrac{1}{2}\coth(x)\sinh(y),\tfrac{1}{2}\coth(x)\cosh(y)\right)
\end{align*}
The first one is a totally umbilical paraboloid, the second one is a hyperbolic cylinder, and the last one is known as the semithrough. The dual minimal surfaces in $\mathrm{Nil}_3(H)$ are of the form $u(x,y)=yf(x)$ for some function $f\in C^\infty(I)$, including  the plane $u(x,y)=0$ and the invariant surface $u(x,y)=Hxy$.
\end{example}

\section{Existence of complete spacelike surfaces}

In this section, we will discuss how the conformal theories of constant mean curvature multigraphs in $\mathbb{E}(\kappa,\tau)$ and $\mathbb{L}(\kappa,\tau)$ agree via the duality. We will begin by studying the relation between being complete and being an entire graph, essentially following the ideas in~\cite{LeeMan}.

\begin{lemma}\label{lemma:nonexistence}
Let $\pi:\mathbb{L}\to M$ be a Lorentzian Killing submersion and let $\Sigma$ be a complete spacelike surface immersed in $\mathbb{L}$. If $M$ is simply connected, then $\Sigma$ is an entire graph.
\end{lemma}

\begin{proof}
The projection $\pi_{|\Sigma}:\Sigma\to M$ is a distance non-decreasing local diffeomorphism, and hence a covering map by~\cite[Lemma~8.1, Ch.~VIII]{KN}. If $M$ is simply connected, then $\pi_{|\Sigma}$ must be one-to-one, i.e., $\Sigma$ must be an entire graph.
\end{proof}

\begin{remark}\label{rmk:completeness}
The converse is not true, since there are non-complete entire spacelike graphs in $\mathbb{L}(\kappa,\tau)$ with any constant mean curvature $H\geq 0$ provided that $\kappa+4\tau^2<0$. Such graphs can be found among the surfaces invariant under a 1-parameter group of isometries, as shown in the case $(\kappa,\tau,H)=(-1,0,0)$ by Albujer~\cite{A} (the same technique works in the general case). 
\end{remark}

This gives an idea of the fact that the behavior of spacelike constant mean curvature surfaces in $\mathbb{L}(\kappa,\tau)$ strongly depends on the sign of $\kappa+4\tau^2$, in a similar fashion as it depends on $\kappa+4H^2$ in the Riemannian case. Next result shows that the theory of complete spacelike surfaces is not interesting when the bundle curvature is supercritical. Note that there is no assumption on the mean curvature of the surface. We will give three different proofs of this result.

\begin{theorem}
If $\kappa+4\tau^2>0$ and $\kappa\leq 0$, then there are no complete spacelike surfaces or entire spacelike graphs in $\mathbb{L}(\kappa,\tau)$.
\end{theorem}

\smallskip\noindent{\it First proof.}\hskip \labelsep
This proof relies on the duality, as well as on a classical trick due to Heinz~\cite{Heinz}. In order to apply Theorem~\ref{thm:duality}, we will assume the mean curvature is constant, though the same argument applies in the general case, see~\cite{LeeMan}. We will assume that $\Sigma\subset\mathbb{L}(\kappa,\tau)$ is an entire graph with constant mean curvature $H$, and reach a contradiction, which will prove the statement in view of Lemma~\ref{lemma:nonexistence}.

The dual graph $F_u$ in $\mathbb{E}(\kappa,H)$ has constant mean curvature $\tau$. Given a bounded domain $\Omega\subset\mathbb{M}^2(\kappa)$ with regular boundary, we can employ the mean curvature equation and the divergence theorem to estimate
\begin{equation}\label{eqn:heinz}
2\tau\mathrm{Area}(\Omega)=\int_\Omega\mathrm{div}\left(\frac{Gu}{\sqrt{1+\|Gu\|^2}}\right)=\int_{\partial\Omega}\frac{\langle Gu,\eta\rangle}{\sqrt{1+\|Gu\|^2}}\leq\mathrm{Length}(\partial\Omega),
\end{equation}
where $\eta$ is a unit outer conormal vector field to $\Omega$ along its boundary, and we have used Cauchy-Schwartz inequality. Taking the infimum in~\eqref{eqn:heinz} over $\Omega$ we get that
\begin{equation}\label{eqn:cheeger}
2\tau\leq\inf\left\{\frac{\mathrm{Length}(\partial\Omega)}{\mathrm{Area}(\Omega)}:\Omega\subset\mathbb{M}^2(\kappa) \text{ bounded and regular}\right\}.
\end{equation}
The \textsc{lhs} in~\eqref{eqn:cheeger} is the so-called Cheeger constant of $\mathbb{M}^2(\kappa)$, which is known to be zero if $\kappa\geq 0$, and $\frac{1}{2}\sqrt{-\kappa}$ if $\kappa\leq 0$. Therefore, Equation~\eqref{eqn:cheeger} yields $\kappa+4\tau^2\leq 0$ if $\kappa\leq 0$, so we are done.
\hfill\penalty10000\raisebox{-.09em}{$\Box$}\par\medskip

\smallskip\noindent{\it Second proof.}\hskip \labelsep
Now we will give a direct argument using the divergence equation for $\tau$. If $\Sigma$ is an entire spacelike graph in $\mathbb{L}(\kappa,\tau)$, parameterized as $F_v$ for some $v\in C^\infty(\mathbb{M}^2(\kappa))$, then the vertical fibers have infinite length and the generalized gradient satisfies $\mathrm{div}(JGv)=\mathrm{div}(J\nabla v+JZ)=\mathrm{div}(JZ)=-2\tau$, and the spacelike property reads $\|Gv\|\leq 1$. If $\Omega\subset\mathbb{M}^2(\kappa)$ is a bounded domain with regular boundary, we can estimate
\[2\tau\mathrm{Area}(\Omega)=\int_\Omega\mathrm{div}(-Gv)=\int_{\partial\Omega}\langle -Gv,\eta\rangle\leq\mathrm{Length}(\partial\Omega),\]
where we have used the Cauchy-Schwartz inequality, and $\eta$ denotes a unit outer conormal vector field to $\Omega$ along its boundary. Then we can conclude with a similar argument as in the first proof.
\hfill\penalty10000\raisebox{-.09em}{$\Box$}\par\medskip

\smallskip\noindent{\it Third proof.}\hskip \labelsep
Let us consider the circle in the $z=0$ plane given by
\[\gamma:\mathbb{R}\to\mathbb{L}(\kappa,\tau),\qquad \gamma(t)=(r\cos(t),r\sin(t),0).\]
From the assumption $\kappa+4\tau^2>0$, it is easy to check that $\gamma$ is a closed timelike curve for some values of $r\in(0,\frac{2}{\sqrt{-\kappa}})$ if $\kappa<0$, or $r\in(0,\infty)$ if $\kappa=0$. This means that $\mathbb{L}(\kappa,\tau)$ is not a causal spacetime, and in particular it is not distinguishing, so it does not admit complete spacelike surfaces (more information about these definitions and the fact that distinguishing spacetimes admit a Killing submersion structure and complete spacelike surfaces can be found in~\cite{JS}).
\hfill\penalty10000\raisebox{-.09em}{$\Box$}\par\medskip

In the rest of this section, we will substitute completeness with a weaker condition in order to apply Riemannian results. However we will require that the mean curvature is constant.

\begin{definition}
A spacelike conformal immersion $\widetilde X:\Sigma\to\mathbb{L}(\kappa,\tau)$ with constant mean curvature $H$ is said to be dual-complete when the dual immersion $X:\Sigma\to\mathbb{E}(\kappa,H)$ is complete.
\end{definition}

Denote by $g$ and $\widetilde g$ the Riemannian metrics in $\Sigma$ that make $X$ and $\widetilde X$ isometric, respectively. Then Equations~\eqref{eqn:1ff} and~\eqref{eqn:angle} yields the conformal relation $g=\widetilde\nu^2\widetilde g$, where $\widetilde\nu$ stands for the angle function of $\widetilde X$. Since $\widetilde\nu\geq 1$, we deduce that any complete spacelike immersion is dual-complete.

\begin{theorem}\label{thm:dual-complete}
Let $\widetilde X:\Sigma\to\mathbb{L}(\kappa,\tau)$ be a dual-complete spacelike conformal immersion with constant mean curvature. Then either $\widetilde X(\Sigma)$ is a horizontal slice in $\mathbb{S}^2(\kappa)\times\mathbb{R}_1$ or $\kappa+4\tau^2\leq 0$.
\begin{enumerate}
 \item If $\kappa+4\tau^2<0$, then $\widetilde X(\Sigma)$ is graph over a simply connected domain bounded by disjoint curves with constant curvature $\pm 2\tau$ in $\mathbb{H}^2(\kappa)$.
 \item If $\kappa+4\tau^2=0$, then $\widetilde X(\Sigma)$ is complete, and hence an entire graph. 
\end{enumerate}
In particular, $\Sigma$ is simply connected.
\end{theorem}

\begin{proof}
Let us consider the dual graph $X:\Sigma\to\mathbb{E}(\kappa,H)$, where $H$ is the mean curvature of $\widetilde X$. Since $X$ defines a complete surface with mean curvature $\tau$, and it is stable since it is nowhere vertical, the arguments in the proof of~\cite[Theorem~3.1]{MPR} show that $X$ is either a horizontal slice in $\mathbb{S}^2\times\mathbb{R}$ or $\kappa+4\tau^2\leq 0$.

If $\kappa+4\tau^2=0$, then the $X$ has critical mean curvature and hence $X(\Sigma)$ is an entire graph (this was proved by Hauswirth, Rosenberg and Spruck~\cite{HRS} for mean curvature $\frac{1}{2}$ surfaces in $\mathbb{H}^2\times\mathbb{R}$, by Daniel and Hauswirth~\cite{DH} for minimal ones in the Heisenberg space, and it extends to the rest of cases by means of the Daniel correspondence, see~\cite{DHM}). The proof that it is complete can be found in~\cite[Theorem~4.6.2]{DHM}. If $\kappa+4\tau^2<0$, then $X$ has subcritical constant mean curvature and it must be a graph over a domain $\Omega\subset\mathbb{H}^2(\kappa)$ bounded by a (possibly empty) family of disjoint curves with constant curvature $\pm 2\tau$ (the sign depends on whether the function defining the graph goes to $+\infty$ or $-\infty$), see~\cite{MR}.
\end{proof}

Observe that this result is sharp, in the sense that there are many entire graphs with any constant mean curvature when $\kappa+4\tau^2=0$. If $\kappa=\tau=H=0$, then we only get affine planes in $\mathbb{L}^3$ due to the well-known Calabi-Bernstein theorem~\cite{Calabi}, but in the rest of cases the results of Wan and Au~\cite{Wan,WanAu} can be used to show that the moduli space of entire graphs with constant mean curvature $H$ (up to ambient isometries) is in one-to-one correspondence with the space of holomorphic quadratic differentials in the complex plane (parabolic case) or the unit disk (hyperbolic case), with the exception that the differential must be non-zero in the parabolic case. This correspondence is given by the Hopf differential, and was the key step in Fern\'{a}ndez and Mira's solution to the Bernstein problem in $\mathrm{Nil}_3(\tau)$~\cite{FM}.

In the case $\kappa+4\tau^2<0$, the family of weak-complete graphs is richer. On the one hand, there are plenty of entire graphs with constant mean curvature (for instance, Nelli and Rosenberg~\cite{NR} constructed entire minimal graphs in $\mathbb{H}^2\times\mathbb{R}$ with arbitrary continuous asymptotic values in the ideal boundary, whose duals are entire spacelike maximal graphs in $\mathbb{H}^2\times\mathbb{R}_1$), though completeness is not waranteed, see Remark~\ref{rmk:completeness}. On the other hand, ideal Scherk graphs with constant mean curvature $0\leq H<\frac{1}{2}$ in $\mathbb{H}^2\times\mathbb{R}$ and $\widetilde{\mathrm{Sl}}_2(\mathbb{R})$ constructed as solutions to certain Jenkins-Serrin problems produce dual graph exhibiting the features in item 1 of Theorem~\ref{thm:dual-complete}. These ideal Scherk graphs are conformally the plane, see~\cite{MN}.

\begin{remark}\label{rmk:daniel}
It is also worth mentioning here relation with Daniel's isometric correspondence for surfaces in $\mathbb{E}(\kappa,\tau)$-spaces~\cite{Dan}. Given $\kappa_1,\tau_1,H_1\in\mathbb{R}$, consider new parameters $\kappa_2,\tau_2,H_2\in\mathbb{R}$ such that
\[\kappa_1-4\tau_1^2=\kappa_2-4\tau_2^2,\qquad \tau_1^2+H_1^2=\tau_2^2+H_2^2.\]
Given a simply connected Riemannian surface $\Sigma$, there is an correspondence between isometric immersions of $\Sigma$ with constant mean curvature $H_1$ in $\mathbb{E}(\kappa_1,\tau_1)$ and isometric immersions of $\Sigma$ with constant mean curvature $H_2$ in $\mathbb{E}(\kappa_2,\tau_2)$.

Since Daniel correspondence preserves the angle function, it also preserves locally the graphical condition (not globally in general), so we can connect both correspondences and get that there is a correspondence between four families of conformal immersions of a simply-connected Riemann surface $\Sigma$:
\begin{enumerate}
 \item[(a)] Conformal immersions $X:\Sigma\to\mathbb{E}(\kappa_1,\tau_1)$ with constant mean curvature $H_1$ and nowhere vertical.
 \item[(b)] Conformal immersions $X':\Sigma\to\mathbb{E}(\kappa_2,\tau_2)$ with constant mean curvature $H_2$ and nowhere vertical.
 \item[(c)] Conformal spacelime immersions $\widetilde X:\Sigma\to\mathbb{L}(\kappa_1,H_1)$ with constant mean curvature $H_1$.
 \item[(d)] Conformal spacelime immersions $\widetilde X':\Sigma\to\mathbb{L}(\kappa_2,H_2)$ with constant mean curvature $H_2$.
\end{enumerate} 
If we call $g$ the pullback metric by $X$ in $\Sigma$, then $X'$ also induces the same metric on $\Sigma$, and both $\widetilde X$ and $\widetilde X'$ induce the metric $\nu^2 g$ on $\Sigma$, so the correspondence between the above items (c) and (d) is also isometric, giving a Lorentzian analogue of Daniel correspondence, see also Palmer's approach~\cite{Palmer}.
\end{remark}

The moral of Remark~\ref{rmk:daniel} is that the conformal theory of constant mean curvature surfaces in the cases (a), (b), (c) and (d) is essentially the same, and yields a beautiful correspondence for holomorphic quadratic differentials and harmonic maps in quite different geometric contexts, e.g., the following harmonic functions are related via this 4-sided correspondence (up to conformal diffeomorphisms):
\begin{itemize}
	\item The classical Gauss map for surfaces with mean curvature $\frac{1}{2}$ in $\mathbb{L}^3$, which takes values in the hyperbolic plane in a natural way.
	\item The classical hyperbolic Gauss map for maximal surfaces in the anti-de Sitter space $\mathbb{H}^3_1(-\frac{1}{2})$, see~\cite{LH,SLee} and the references therein.
	\item The left-invariant Gauss map of minimal surfaces in $\mathrm{Nil}_3(\frac{1}{2})$, see~\cite{Dan2}.
	\item The hyperbolic Gauss map of mean curvature $\frac{1}{2}$ surfaces in $\mathbb{H}^2\times\mathbb{R}$, see~\cite{FM}.
\end{itemize}
More information about these Gauss maps can be found in~\cite{DFM}. On the other hand, the holomorphic quadratic differentials associated with these harmonic maps also coincide, namely the Hopf differential of constant mean curvature in $\mathbb{L}^3$ or $\mathbb{H}_1^3(\kappa)$, and the Abresch-Rosenberg differential~\cite{AR} of minimal surfaces in $\mathrm{Nil}_3(\tau)$ or mean curvature $\frac{1}{2}\sqrt{-\kappa}$ surfaces in $\mathbb{H}^2(\kappa)\times\mathbb{R}$. In~\cite{DFM} the whole family of harmonic Gauss maps of critical mean curvature surfaces in $\mathbb{E}(\kappa,\tau)$ is studied, though this becomes transparent via the duality since all of them correspond to the classical Gauss map in $\mathbb{L}^3$ or the hyperbolic Gauss map in $\mathbb{H}_1^3(\kappa)$.

\section{Estimates for entire minimal graphs in Heisenberg space}

The main goal of this section is to illustrate how we can translate a Lorentzian property into the Riemannian setting, namely we will obtain gradient estimates for entire minimal graphs in $\mathrm{Nil}_3(\tau)$ following the arguments in~\cite{MN}. 

Cheng and Yau~\cite{CY2} proved that the support function $\Phi(x,y,z)=x^2+y^2-z^2$ satisfies the following gradient estimate in a complete spacelike surface in $\mathbb{L}^3$:
\begin{equation}\label{eqn:cy}
\|\widehat\nabla\Phi\|^2\leq C(1+\Phi)^2\leq C(1+r^2)^2,
\end{equation}
for some constant $C>0$, where $\widehat\nabla$ denotes the gradient in the surface and $r=(x^2+y^2)^{1/2}$ is the distance to the origin in the base surface $\mathbb{R}^2$. Since the surface is an entire graph by Lemma~\ref{lemma:nonexistence}, we can parameterize it as $F_v$ for some $v\in C^\infty(\mathbb{R}^2)$. The gradient in~\eqref{eqn:cy} can be worked out as the tangent part of the ambient gradient $\overline\nabla\Phi$, so we deduce that
\begin{equation}\label{eqn:L3-estimate1}
\|\widehat\nabla\Phi\|^2=\|\overline\nabla\Phi\|^2+\langle\overline\nabla\Phi,N\rangle^2\geq \langle\overline\nabla\Phi,N\rangle^2=\frac{4(v-xv_x-yv_y)^2}{1-|\nabla v|^2},
\end{equation}
where $N=(1-|\nabla v|^2)^{-1/2}(v_x\partial_x+v_y\partial_y+\partial_z)$ is the upward-pointing unit normal to the surface, and $\nabla v$ is the usual gradient in $\mathbb{R}^2$.

\begin{lemma}\label{lemma:gradient-estimate-L3}
Let $F_v$ be an entire spacelike graph in $\mathbb{L}^3$ with positive constant mean curvature. Then there is a constant $A>0$ such that
\[|\nabla v|^2\leq 1-\frac{A}{(1+r^2)^2}.\]
\end{lemma}

\begin{proof}
In view of~\eqref{eqn:cy} and~\eqref{eqn:L3-estimate1}, it will be enough to show that there is a constant $M>0$ such that $(v-xv_x-yv_y)^2\geq M$ whenever $x^2+y^2>1$. First, we can apply a translation in $\mathbb{L}^3$ such that the origin lies in the surface, but no straight line through the origin is contained in the surface. Were this not possible, the surface would be ruled, and then the surface would be a hyperbolic cylinder, see~\cite{DVVW}, i.e., up to an isometry of $\mathbb{L}^3$ the surface would be given by $v(x,y)=\frac{1}{2H}\sqrt{1+4H^2x^2}$, and the statement follows. In order to estimate $w=v-xv_x-yv_y$, we will use the fact that the surface is the boundary a convex set as proved by Treibergs~\cite{Treibergs}. Hence, up to a mirror reflection with respect to $z=0$, we can also assume that $v$ is a convex function.

Observe that the intersection of the $z$-axis and the tangent line to the surface at $(x,y,v)$ in the direction of the tangent vector $(x,y,xv_x+yv_y)$ is precisely the point $(0,0,w)$. Since the surface is convex and does not contain a line through the origin, we get that
\begin{equation}\label{eqn:L3-estimate2}
w(x,y)\leq w\left(\frac{x}{\sqrt{x^2+y^2}},\frac{y}{\sqrt{x^2+y^2}}\right)<0,\qquad\quad \text{if }x^2+y^2\geq 1.
\end{equation}
Since $w$ is continuous and the unit circle is compact, the existence of the desired constant $M$ follows from equation~\eqref{eqn:L3-estimate2}.
\end{proof}

\begin{theorem}\label{last-area-estimate}
Let $F_u$ be an entire minimal graph in $\mathrm{Nil}_3(\tau)$.
\begin{itemize}
 \item[(a)] There exists a constant $B>0$ such that $|Gu|\leq B(1+r^2)$,
 \item[(b)] There exists a constant $C>0$ such that $|u|\leq C(1+r^2)^{3/2}$.
\end{itemize}
\end{theorem}

\begin{proof}
The dual graph $F_v$ is an entire spacelike graph in $\mathbb{L}^3$ with constant mean curvature $\tau$. Since these surfaces have reciprocal angle functions, from Lemma~\ref{lemma:gradient-estimate-L3} we get that there exists $A>0$ such that $1+|Gu|^2=(1-|\nabla v|^2)^{-1}\leq A^{-1}(1+r^2)^2<1+A^{-1}(1+r^2)^2$, and we get item (a) by just taking $B=A^{-1/2}$.

Applying the Minkowski inequality to the expression $\nabla u=Gu-Z$, where $Z=-\tau y\partial_x+\tau x\partial_y$, we get that  $|\nabla u|\leq |Gu|+|Z|\leq B(1+r^2)+\tau r$. Hence $|\nabla u|$ grows at most quadratically in $r$, from where it is easy to see that there exists a constant $C>0$ satisfying item (b). 
\end{proof}

Let $F_u$ in $\mathbb{E}(\kappa,\tau)$ and $F_v$ in $\mathbb{L}(\kappa,\tau)$ be dual graphs over a domain $\Omega\subset\mathbb{M}^2(\kappa)$. It follows from~\eqref{eqn:1ff} that the absolute value of the determinant of the Jacobian of $\pi:\mathbb{E}(\kappa,\tau)\to\mathbb{M}^2(\kappa)$ or $\widetilde\pi:\mathbb{L}(\kappa,\tau)\to\mathbb{M}^2(\kappa)$ restricted to $F_u$ or $F_v$ coincides with the angle function $\nu$ or $\widetilde\nu$. Using the change of variables formula, if $f$ is a positive measurable function on $\Omega$, we get that
\begin{equation}\label{eqn:integration}
\int_{F_u} (f\circ\pi)\nu=\int_{\Omega}f=\int_{F_v}(f\circ\widetilde\pi)\widetilde\nu.
\end{equation}
In particular, the area of $\Omega$ is equal to the integral of $\nu$ or $\widetilde\nu$ over the surface. In the particular case of minimal graphs in $\mathrm{Nil}_3(\tau)$, this implies that $\nu$ is not integrable since $\mathbb{R}^2$ has infinite area. The estimates given by Theorem~\ref{last-area-estimate} show that $\nu$ is not square integrable either.

\begin{corollary}~
\begin{enumerate}
 \item[(a)] The angle function of an entire minimal graph in $\mathrm{Nil}_3(\tau)$ is not square-integrable.
 \item[(b)] An entire graph with positive constant mean curvature in $\mathbb{L}^3$ has infinite area.
\end{enumerate} 
\end{corollary} 

\begin{proof}
Plugging $f\circ\pi=\nu$ in Equation~\eqref{eqn:integration}, it will be enough to show that $\nu$ is not integrable in $\Omega$. By means of Theorem~\ref{last-area-estimate}, we get that $\nu=(1+|Gu|^2)^{1/2}\leq B'(1+r^2)$ for some $B'>0$. Integrating in polar coordinates in a disk $D_R\subset\mathbb{R}^2$ of radius $R$ centered at the origin, we get that
\[\int_{D_R}\nu=\int_{D_R}\frac{1}{\sqrt{1+|Gu|^2}}\geq\int_0^R\frac{2\pi r\,\mathrm{d} r}{B'(1+r^2)}.\]
Since the last integral diverges, we are done.
\end{proof}  

\section{Other results}

\subsection{Behaviour with respect to isometries}

Lee showed that translating and rotating a surface in $\mathrm{Nil}_3(\tau)$ corresponds to translating and rotating the dual surface in $\mathbb{L}^3$, respectively~\cite{Lee}. Here we will present this result from a more abstract point of view that applies to all $\mathbb{E}(\kappa,\tau)$-spaces (and also works in the Killing submersion setting). Problems related to the congruence of dual surfaces have been already treated in the literature, see~\cite{AL,KL}.

Let $\mathrm{Iso}_+(\mathbb{E}(\kappa,\tau),\xi)$ be the group of direct isometries of $\mathbb{E}(\kappa,\tau)$ preserving the unit Killing vector field $\xi$ and the orientation. It is worth pointing out that, if $\kappa-4\tau^2\neq 0$, then any isometry $T$ of $\mathbb{E}(\kappa,\tau)$ satisfies $T_*\xi=\pm\xi$, whereas any isometry is direct if $\tau\neq 0$. The group $\mathrm{Iso}_+(\mathbb{L}(\kappa,\tau),\xi)$ is defined likewise. We are not considering the case $T_*\xi=-\xi$ due to the fact that such isometries change the sign of the mean curvature of the graphs.

For any $T\in\mathrm{Iso}_+(\mathbb{E}(\kappa,\tau),\xi)$, there is a direct isometry $h\in\mathrm{Iso}_+(\mathbb{M}^2(\kappa))$ such that the following diagram is commutative
\[\begin{tikzcd}
 \mathbb{E}(\kappa,\tau)
 	\arrow[rr, rightarrow, "T"]&&
 \mathbb{E}(\kappa,\tau)
    \arrow[d, rightarrow,"\pi"]\\
 \mathbb{M}^2(\kappa)
 	\arrow[u, leftarrow, "\pi"]
 	\arrow[rr, rightarrow, "h"]&&
 \mathbb{M}^2(\kappa)
\end{tikzcd}\]
This association is a group morphism $\mathrm{Iso}_+(\mathbb{E}(\kappa,\tau),\xi)\to\mathrm{Iso}_+(\mathbb{M}^2(\kappa))$ whose kernel is the subgroup of vertical translations. Conversely, given $h\in\mathrm{Iso}_+(\mathbb{M}^2(\kappa))$, there is $T\in\mathrm{Iso}_+(\mathbb{E}(\kappa,\tau),\xi)$ making the diagram commutative, and such a $T$ is unique up to vertical translations. The proof of these results can be found in~\cite{Man}, and similar arguments work in the Lorentzian setting. Therefore we can define a bijective map
\[R:\mathrm{Iso}_+(\mathbb{E}(\kappa,\tau),\xi)\to\mathrm{Iso}_+(\mathbb{L}(\kappa,\tau),\xi),\]
such that, for each $T\in\mathrm{Iso}_+(\mathbb{E}(\kappa,\tau),\xi)$, the image $R(T)$ is the only isometry in $\mathrm{Iso}_+(\mathbb{L}(\kappa,\tau),\xi)$ that projects to the same isometry in $\mathrm{Iso}_+(\mathbb{M}^2(\kappa))$ as $T$, and such that $R(T)(0,0,0)=T(0,0,0)$.

\begin{proposition}\label{prop:transform}
Let $X:\Sigma\to\mathbb{E}(\kappa,\tau)$ and $\widetilde X:\Sigma\to\mathbb{L}(\kappa,H)$ be dual conformal immersions, and let $T\in\mathrm{Iso}_+(\mathbb{E}(\kappa,\tau),\xi)$. Then $T\circ X$ and $R(T)\circ\widetilde X$ are dual conformal immersions.
\end{proposition}

\begin{proof}
Let $h\in\mathrm{Iso}_+(\mathbb{M}^2(\kappa))$ the isometry to which both $T$ and $R(T)$ project in the base surface. Since the computation can be localized, let us assume that $X$ and $\widetilde X$ are given as graphs $F_u$ and $F_v$ over a simply connected domain $\Omega\subset\mathbb{M}^2(\kappa)$. Let $F_{\overline u}$ and $F_{\overline v}$ be the graphs over $h(\Omega)$ associated to $T\circ F_u$ and $R(T)\circ F_v$, respectively. From Equation~\eqref{eqn:N}, where $N$ is a unit normal to $F_u$, and the fact that $T_*N$ is an unit normal to $F_{\overline u}$, the condition $\pi\circ T=h\circ\pi$ allows us to work out
\begin{equation}\label{eqn:transform1}
\frac{G\overline{u}}{\sqrt{1+\|G\overline u\|^2}}=\pi_*(T_*N)=h_*(\pi_*N)=\frac{h_*G{u}}{\sqrt{1+\|Gu\|^2}}.
\end{equation}
Taking squared norms in~\eqref{eqn:transform1}, we easily reach $\|Gu\|=\|G\overline u\|$ via the isometry $h$. We deduce that $G\overline u=h_*Gu$, and likewise we can check that $\widetilde G\overline v=h_*\widetilde Gv$. Therefore it suffices to apply $h_*$ to Equation~\eqref{thm:duality:eqn3} to realize that $F_{\overline u}$ and $F_{\overline v}$ also satisfy the same twin relations, so they are dual graphs.
\end{proof}

Nonetheless, not all the isometries preserve the Killing vector field, and there are cases where this situation becomes specially interesting. In the case the surface has critical constant mean curvature, the dual surface lies in a space form, whose isometry group has dimension 6 (two extra dimensions of isometries not preserving the Killing direction). This leads to non-trivial 2-parameter deformation of critical mean curvature surfaces. More precisely:
\begin{enumerate}
 	\item There is a 2-parameter deformation of constant mean curvature $H$ surfaces in $\mathbb{E}(\kappa,\tau)$ with $\kappa+4H^2=0$. It corresponds to hyperbolic and parabolic rotations in $\mathbb{L}^3$ or in the anti-de Sitter space $\mathbb{H}^3_1(\kappa)$.
 	\item There is a 2-parameter deformation of spacelike constant mean curvature $H$ surfaces in $\mathbb{L}(\kappa,\tau)$ with $\kappa-4H^2=0$. It corresponds to rotations with respect to non-vertical axes in $\mathbb{R}^3$ or in the round sphere $\mathbb{S}^3(\kappa)$.
 \end{enumerate}
Proposition~\ref{prop:transform} shows that the above items 1 and 2 reflect all possible non trivial actions of isometries.

\subsection{Hessian equations}

Here we will present a classical way of constructing solutions to the Hessian-one equation $f_{xx}f_{yy}-f_{xy}^2=1$ in Euclidean plane by means of solutions to the minimal surface equation in $\mathbb{R}^3$ (or, equivalently, solutions to the maximal surface equation in $\mathbb{L}^3$). The technique we will explain below, was pointed out by Calabi~\cite{Calabi}, though the same ideas had already been used before. For instance, the same arguments were applied by Osserman~\cite{Osserman} to classify all entire minimal 2-dimensional graphs of the form $(x, y, f(x,y), g(x,y))$ in $\mathbb{R}^4$, see~\cite{Lee3}. Also Nitsche~\cite{Nitsche}, following Heinz' idea, solved the classical Bernstein problem in $\mathbb{R}^3$ by means of entire solutions to the Hessian-one equation.

\begin{lemma}\label{lemma:identities}
Let $F_u$ be a graph in $\mathbb{E}(\kappa,\tau)$ with mean curvature $H$, not necessarily constant. Then:
\begin{align*}
 \frac{1}{\lambda_\kappa^2}\left[\left(\frac{\lambda_\kappa^2(1+\alpha^2)}{\omega}\right)_y-\left(\frac{\lambda_\kappa^2\alpha\beta}{\omega}\right)_x\right]&=\frac{2\tau\alpha}{\omega}-2H\beta+\frac{(\lambda_\kappa)_y}{\lambda_\kappa}\frac{1+\omega^2}{\omega},\\
 \frac{1}{\lambda_\kappa^2}\left[\left(\frac{\lambda_\kappa^2(1+\beta^2)}{\omega}\right)_x-\left(\frac{\lambda_\kappa^2\alpha\beta}{\omega}\right)_y\right]&=-\frac{2\tau\beta}{\omega}-2H\alpha+\frac{(\lambda_\kappa)_x}{\lambda_\kappa}\frac{1+\omega^2}{\omega}.
\end{align*}
\end{lemma}

\begin{proof}
It is a long but straightforward computation to check that the following formulas hold true (it suffices to do the derivatives in the \textsc{lhs} in terms of $u$ and gather the resulting terms, taking into account that the second-order terms must be collected into $H$):
\begin{equation}\label{lema:Hessian:eqn1}
\begin{aligned}
 \left(\frac{1+\alpha^2}{\omega}\right)_y-\left(\frac{\alpha\beta}{\omega}\right)_x&=\frac{2\tau\alpha}{\omega}-2H\beta-\kappa\,\frac{\lambda_\kappa}{\omega}\left(x\alpha\beta-\tfrac{y}{2}(\alpha^2-\beta^2)\right),\\
 \left(\frac{1+\beta^2}{\omega}\right)_x-\left(\frac{\alpha\beta}{\omega}\right)_y&=-\frac{2\tau\beta}{\omega}-2H\alpha-\kappa\,\frac{\lambda_\kappa}{\omega}\left(\tfrac{x}{2}(\alpha^2-\beta^2)+y\alpha\beta\right).
\end{aligned}
\end{equation}
Since the derivatives of $\lambda_\kappa$ satisfy $(\lambda_\kappa)_x=-\frac{\kappa}{2}x\lambda_\kappa^2$ and $(\lambda_\kappa)_y=-\frac{\kappa}{2}y\lambda_\kappa^2$, we can use them to eliminate $x$ and $y$ in~\eqref{lema:Hessian:eqn1} and reach the following two identities:
\begin{align*}
 \left(\frac{1+\alpha^2}{\omega}\right)_y-\left(\frac{\alpha\beta}{\omega}\right)_x&=\frac{2\tau\alpha}{\omega}-2H\beta+2\frac{(\lambda_\kappa)_x}{\lambda_\kappa}\frac{\alpha\beta}{\omega}-\frac{(\lambda_\kappa)_y}{\lambda_\kappa}\frac{\alpha^2-\beta^2}{\omega},\\
 \left(\frac{1+\beta^2}{\omega}\right)_x-\left(\frac{\alpha\beta}{\omega}\right)_y&=-\frac{2\tau\beta}{\omega}-2H\alpha +\frac{(\lambda_\kappa)_x}{\lambda_\kappa}\frac{\alpha^2-\beta^2}{\omega}+2\frac{(\lambda_\kappa)_y}{\lambda_\kappa}\frac{\alpha\beta}{\omega}.
\end{align*}
The identities in the statement follow from grouping the different derivatives.
\end{proof}

The formulas in the statement may seem artificial, but their \textsc{lhs} are related to $\mathrm{div}(Je_1)$ and $\mathrm{div}(Je_2)$, where $\{e_1,e_2\}$ is the frame given by Equation~\eqref{eqn:tangent-frame}, the divergence is computed in $\mathbb{M}^2(\kappa)$, and $J$ is a $\frac{\pi}{2}$-rotation in the tangent bundle.

If $\kappa=\tau=H=0$, then $\lambda_\kappa\equiv 1$ and the aforementioned formulas have been obtained in the literature by means of different techniques, see~\cite{Osserman,Lee4}. From Lemma~\ref{lemma:identities} we get that
\begin{equation}\label{eqn:Hessian1}
\begin{aligned}
 \left(\frac{1+\alpha^2}{\omega}\right)_y-\left(\frac{\alpha\beta}{\omega}\right)_x&=0,\\
 \left(\frac{1+\beta^2}{\omega}\right)_x-\left(\frac{\alpha\beta}{\omega}\right)_y&=0.
\end{aligned}
\end{equation}
If $u$ is defined on a simply-connected domain $\Omega\subseteq\mathbb{R}^2$, then Poincar\'{e}'s Lemma guarantees the existence of $g,h\in C^\infty(\Omega)$ such that~\eqref{eqn:Hessian1} gives
\begin{equation}\label{eqn:Hessian2}
\frac{1+\alpha^2}{\omega}=g_y,\quad \frac{\alpha\beta}{\omega}=g_x,\quad\frac{1+\beta^2}{\omega}=h_y,\quad\frac{\alpha\beta}{\omega}=h_x.
\end{equation}
Now the second and third equations in~\eqref{eqn:Hessian2} imply that $g_x=h_y$, so again Poincar\'{e}'s Lemma yields the existence of $f\in C^\infty(\Omega)$ such that $g=f_y$ and $h=f_x$. Thus, the identities in~\eqref{eqn:Hessian2} can be rewritten as
\begin{align*}
f_{yy}&=\frac{1+\alpha^2}{\omega},&
f_{xy}&=\frac{\alpha\beta}{\omega},&
f_{xx}&=\frac{1+\beta^2}{\omega}.
\end{align*}
and hence the function $f$ satisfies
\[f_{xx}f_{yy}-f_{xy}^2=\frac{1+\beta^2}{\omega}\cdot\frac{1+\alpha^2}{\omega}-\frac{\alpha^2\beta^2}{\omega^2}=1.\]

\begin{remark}
There are further relations that connect these equations with others that will not be explored here. As a sample, any solution of $f_{xx} f_{yy} - f_{xy}^2 =1$ satisfies the property that the gradient map $(x,y)\mapsto (x,y,f_x,f_y)$ is a parametrization of a minimal surface in $\mathbb{R}^4$, or equivalently, a special Lagrangian surface in $\mathbb{C}^2$ or a holomorphic curve in $\mathbb{C}^2$ for some complex structure, see~\cite{Lee3}.
\end{remark}

\section{Open questions}

In this last section, we will pose three questions that arose in the above discussions, and that could be of interest in a further development of twin correspondences.

\begin{question}\label{q1}
Let $M$ be a Riemannian surface, and let $H\in C^\infty(M)$ be an arbitrary function (with $\int_MH=0$ if $M$ is compact). If $\pi:\mathbb{L}\to M$ is a Lorentzian Killing submersion over $M$ whose fibers have infinite length, is there an entire graph in $\mathbb{L}$ with prescribed mean curvature $H$?
\end{question}

This question is equivalent (if $M$ is simply connected) to that of finding entire minimal graphs in a Riemannian Killing submersion over $M$ with prescribed bundle curvature, see~\cite{LeeMan}, and the author conjectures that the solution to Question~\ref{q1} is affirmative. This is the case provided that the base surface is a sphere (even if the Killing vector field is non-unitary), see~\cite{LerMan}. Nonetheless, this question remains unsolved in the case of the Lorentz-Minkowski space $\mathbb{L}^3$. The condition of fibers with infinite length ensures that the submersion admits entire sections.

\begin{question}\label{q2}
Calabi associated solutions of the minimal surface equation in $\mathbb{R}^3$ with solutions of the Hessian-one equation. Is it possible to generalize this idea to associate solutions of the constant mean curvature equation in $\mathbb{E}(\kappa,\tau)$ with solutions to another natural non-linear \textsc{pde}?
\end{question}

\begin{question}\label{q3}
Let $\Sigma$ be a constant mean curvature surface in $\mathbb{R}^3$ which is a graph over a plane with zero boundary values and intersecting the plane orthogonally, and let us consider the dual maximal surface in $\mathrm{Nil}_3^1(\tau)$. Is it always possible to extend this surface as a zero mean curvature surface (not spacelike) beyond the boundary of the domain? 
\end{question}

In the cases we are able to work it out (see Examples~\ref{example1} and~\ref{example2}), it seems that the dual of such a surface is extended beyond the lightlike boundary, so this property could be understood as a dual Schwarz reflection principle.


\begin{thebibliography}{99}

\bibitem{AR}
U. Abresch, H. Rosenberg.
Generalized Hopf differentials.
\emph{Mat. Contemp.} \textbf{28} (2005), 1--28.

\bibitem{A} 
A. L. Albujer. New examples of entire maximal graphs in $\mathbb{H}^2\times\mathbb{R}_1$.
\emph{Differential Geom. Appl.} \textbf{26} (2008), no. 4, 456--462.

\bibitem{AA}
A. L. Albujer, L. J. Al\'{i}as. 
Calabi-Bernstein results for maximal surfaces in Lorentzian product spaces. 
\emph{J. Geom. Phys.} \textbf{59} (2009), no. 5, 620--631.

\bibitem{AL}
H. Ara\'{u}jo, M. L. Leite.
How many maximal surfaces do correspond to one minimal surface? 
\emph{Math. Proc. Camb. Phil. Soc.} \textbf{146} (2009), 165--175. 

\bibitem{Calabi}
E. Calabi.
Examples of Bernstein problems for some non-linear equations, 
\emph{Proc. Sympos. Pure Math.} \textbf{15} (1970), Amer. Math. Soc., Providence, RI, 223--230.
 
\bibitem{Catalan}
E. Catalan.
M\'{e}moire sur les surfaces dont les rayons de courbure, en chaque point, sont \'{e}gaux et de signes contraires.
\emph{C. R. Acad. Sci. Paris} \textbf{41} (1855), 1019--1023.

\bibitem{CY2} S. Y. Cheng, S. T. Yau.
\newblock Maximal space-like hypersurfaces in the Lorentz-Minkowski spaces.
\newblock \emph{Ann. of Math. (2)} \textbf{104} (1976), no. 3, 407--419.

\bibitem{Dan}
B. Daniel. Isometric immersions into 3-dimensional homogeneous manifolds. 
\emph{Comment. Math. Helv.} \textbf{82} (2007), no. 1, 87--131.

\bibitem{Dan2}
B. Daniel. The Gauss map of minimal surfaces in the Heisenberg group. 
\emph{Int. Math. Res. Not.} \textbf{2011} (2011), no. 3, 674--695.

\bibitem{DFM}
B. Daniel, I. Fern\'{a}ndez, P. Mira.
The Gauss map of surfaces in $\widetilde{\mathrm{PSL}}_2(\mathbb{R})$.
\emph{Calc. Var.} \textbf{52} (2015), 507--528.

\bibitem{DH}
B. Daniel, L. Hauswirth. 
Half-space theorem, embedded minimal annuli and minimal graphs in the Heisenberg group. 
\emph{Proc. Lond. Math. Soc. (3)} \textbf{98} (2009), no. 2, 445--470.

\bibitem{DHM}
B. Daniel, L. Hauswirth, P. Mira. 
\emph{Lecture notes on homogeneous 3--manifolds}. 4th KIAS workshop on Differential Geometry, Seoul, 2009.

\bibitem{DVVW} F. Dillen, I. Van der Woewtyne, L. Vestraelen, J. Walrave.
\newblock Ruled surfaces of constant mean curvature in $3$-dimensional Minkowski space.
\newblock Geometry and topology of submanifolds, VIII, 145--147, World Sci. Publ., River Edge, 1996.

\bibitem{FM}
I. Fern\'{a}ndez, P.~Mira.
Holomorphic quadratic differentials and the Bernstein problem in Heisenberg space.
\emph{Trans. Amer. Math. Soc.} \textbf{361} (2009), 5737--5752.

\bibitem{HRS}
L. Hauswirth, H. Rosenberg, J. Spruck. On complete mean curvature $\frac{1}{2}$-surfaces in $\mathbb{H}^2\times\mathbb{R}$. \emph{Commun. Anal. Geom.} {\bf 16} (2008), 989--1005.

\bibitem{Heinz}
E. Heinz. \"{U}ber Fl\"{a}chen mit eineindeutiger Projektion auf eine Ebene, deren Kr\"{u}mmungen durch Ungleichungen eingeschr\"{a}nkt sind.
\emph{Math. Ann.} \textbf{129} (1955), 451--454.

\bibitem{JS}
M. A. Javaloyes, M. S\'{a}nchez. A note on the existence of standard splittings for conformally stationary spacetimes, \emph{Class. Quantum. Grav.} \textbf{25} (2008), 168001, 7pp.

\bibitem{KL}
S. Kaya, R. L\'{o}pez.
On the duality between rotational minimal surfaces and maximal surfaces.
Preprint available at arXiv:1703.04018.

\bibitem{KN}
S. Kobayashi, K. Nomizu, \emph{Foundations of Differential Geometry}, Vol. II, Interscience, New York, 1969.

\bibitem{Lee}
H. Lee. 
Extensions of the duality between minimal surfaces and maximal surfaces.
\emph{Geom. Dedicata} \textbf{151} (2011), 373--386.

\bibitem{Lee2}
H. Lee.
Maximal surfaces in Lorentzian Heisenberg space.
\emph{Differential Geom. Appl.} \textbf{29} (2011), no. 1, 73--84.

\bibitem{Lee3}
H. Lee.
Minimal surface systems, maximal surface systems and special Lagrangian equations.
\emph{Trans. Amer. Math. Soc.} \textbf{365} (2013), no. 7, 3775--3797.

\bibitem{Lee4}
H. Lee.
Minimal surface system in Euclidean four-space.
Preprint available at arXiv:1706.05751.

\bibitem{LeeMan}
H. Lee, J. M. Manzano. Generalized Calabi correspondence and complete spacelike surfaces. {\it Asian J. Math.} (to appear).

\bibitem{SLee}
S. Lee.
Timelike surfaces of constant mean curvature $\pm 1$ in anti-de Sitter 3-space $\mathbb{H}^3_1(-1)$.
\emph{Ann. Global Anal. Geom.} \textbf{29} (2006), 361--407.

\bibitem{LerMan}
A. Lerma, J. M. Manzano. Compact stable surfaces with constant mean curvature in Killing submersions. {\it Ann. Mat. Pura Appl.} (to appear).

\bibitem{LH}
J. H. S. de Lira, J. A. Hinojosa.
The Gauss map of minimal surfaces in the Anti-de Sitter space.
\emph{J. Geom. Phys.} \textbf{61} (2011), 610--623.

\bibitem{LLS}
F. J. L\'{o}pez, R. L\'{o}pez, R. Souam.
Maximal surfaces of Riemann type in Lorentz-Minkowski space $\mathbb{L}^3$.
\emph{Michigan Math. J.} \textbf{47} (2000), 469--497.

\bibitem{Man}
J. M. Manzano. On the classification of Killing submersions and their isometries. {\it Pac. J. Math.} {\bf 270} (2014), no. 2, 367--392.

\bibitem{MPR}
J. M. Manzano, J. P\'{e}rez, M. M. Rodr\'{i}guez. Parabolic stable surfaces with constant mean curvature. \emph{Calc. Var.} \textbf{42} (2011), no. 1--2, 137--152.

\bibitem{MN}
J. M. Manzano, B. Nelli.
Height and area estimates for constant mean curvature graphs in $\mathbb{E}(\kappa,\tau)$-spaces.
\emph{J. Geom. Anal.} (to appear).

\bibitem{MR}
J. M. Manzano, M. M. Rodr\'{i}guez.
On complete constant mean curvature vertical multigraphs in $\mathbb{E}(\kappa,\tau)$.
\emph{J. Geom. Anal.} \textbf{25} (2015), no. 1, 336--346.

\bibitem{MMPR}
W. H. Meeks, P. Mira, J. P\'{e}rez, A. Ros.
Constant mean curvature spheres in homogeneous three-manifolds.
Preprint available at arXiv:1706.09394.

\bibitem{NR}
B. Nelli, H. Rosenberg.
Minimal surfaces in $\mathbb{H}^2\times\mathbb{R}$.
\emph{Bull. Braz. Math. Soc.} \textbf{33} (2002), no. 2, 263--292.

\bibitem{Nitsche}
J. Nitsche.
Elementary proof of Bernstein's theorem on minimal surfaces.
\emph{Ann. of Math.} \textbf{66} (1957), no. 3, 543--544.

\bibitem{Osserman}
R. Osserman.
\emph{A survey of minimal surfaces}, Second edition, Dover Publications, Inc., New
York, 1986.

\bibitem{Palmer}
B. Palmer.
Spacelike constant mean curvature surfaces in pseudo-Riemannian space forms.
\emph{Ann. Global Anal. Geom.} \textbf{8} (1990), no. 3, 217--226.

\bibitem{Treibergs} A. Treibergs.
Entire Spacelike Hypersurfaces on Constant Mean Curvature in Minkowski Space.
\emph{Invent. Math.} \textbf{66} (1982), no. 1, 39--56.
 
\bibitem{Wan}
T. Y. Wan.
Constant mean curvature surface harmonic map and universal Teichmuller space. 
\emph{J. Differential Geom.} \textbf{35} (1992), 643--657.

\bibitem{WanAu}
T. Y. Wan, T. K. Au.
Parabolic constant mean curvature spacelike surfaces, 
\emph{Proc. Amer. Math. Soc.} \textbf{120} (1994), 559--564.

\end{thebibliography}
\end{document}